\documentclass[12pt]{amsart}
\usepackage{graphicx}
\newtheorem{thm}{Theorem}[section]
\newtheorem{cor}[thm]{Corollary}
\newtheorem{lem}[thm]{Lemma}
\newtheorem{prop}[thm]{Proposition}

\theoremstyle{definition}
\newtheorem{defn}[thm]{Definition}
\theoremstyle{remark}
\newtheorem{rem}[thm]{Remark}
\numberwithin{equation}{section}
\newtheorem{conj}[thm]{Conjecture}
\newcommand{\norm}[1]{\left\Vert#1\right\Vert}
\newcommand{\abs}[1]{\left\vert#1\right\vert}

\def\cf{\mathcal{ F}}

\def\bn{{\mathbb N}}

\def\br{{\mathbb R}}

\def\d{\delta}

 \def\O{\Omega}

\def\eb{{\mathbf{e}}}
\def\xb{{\mathbf{x}}}
\def\yb{{\mathbf{y}}}

\def\yb{{\mathbf{y}}}

\def\ifn{\infty}

\def\bal1{\textbf{B}_1^+}

\begin{document}
\title[Surjective QSO]
{PROJECTIVE SURJECTIVITY OF QUADRATIC STOCHASTIC OPERATORS ON $L_1$ AND ITS APPLICATION}

 \author[F. Mukhemedov]{Farrukh Mukhamedov$^*$}
\address{Farrukh Mukhamedov\\
 Department of Mathematical Sciences\\
College of Science, The United Arab Emirates University\\
P.O. Box, 15551, Al Ain\\
Abu Dhabi, UAE} \email{{\tt far75m@gmail.com}, {\tt
farrukh.m@uaeu.ac.ae}}

\author[O. Khakimov]{Otabek Khakimov}
\address{Otabek Khakimov\\
Department of Algebra and Its Applications\\
Institute of Mathematics\\
 P.O. Box, 100125, Tashkent\\
Uzbekistan} \email{{\tt hakimovo@mail.ru}}

\author[A.F. Embong]{Ahmad Fadillah Embong}
\address{Ahmad Fadillah Embong\\
 Department of Mathematical Sciences\\
Faculty of Science, Universiti Teknologi Malaysia\\
 81310 Johor Bahru, Malaysia} \email{{\tt ahmadfadillah.90@gmail.com}}

\thanks{$^*$ Corresponding author}

\maketitle

\begin{abstract}
A nonlinear Markov chain is a discrete time stochastic process
whose transitions depend on both the current state and the current
distribution of the process. The nonlinear Markov chain over a
infinite state space can be identified by a continuous mapping
(the so-called nonlinear Markov operator) defined on a set of all
probability distributions (which is a simplex). In the present
paper, we consider a continuous analogue of the mentioned mapping acting on $L^1$-spaces.
Main aim of the current paper is to investigate projective surjectivity of quadratic stochastic operators (QSO) acting on the set of
all probability measures. To prove the main result, we study the surjectivity of
infinite dimensional nonlinear Markov
operators and apply them to the projective surjectivity of a QSO. Furthermore, the obtained result has been applied for the
existence of positive solution of some  Hammerstein integral equations. 
\end{abstract}

\section{Introduction}

Recently nonlinear Markov chains become an interesting subject in
many areas of applied mathematics. These chains are discrete time
stochastic processes whose transitions, which are defined by
stochastic hypermatrix $\mathcal
P=(P_{i_1,\dots,i_m,k})_{i_1,\dots,i_m,k\in I}$, where
$I\subset\bn$ 
depend on both the current state and the current distribution
of the process \cite{Kol}. These processes were introduced by
McKean \cite{Mc} and have been extensively studied in the context
of the nonlinear Chapman- Kolmogorov equation \cite{6} as well as
the nonlinear Fokker-Planck equation \cite{F}. On the other hand,
we stress that such types of chains are generated by tensors
(hypermatrices), therefore, this topic is closely related to the
geometric and algebraic structures of tensors which have been
systematically studied and has  wide applications in scientific
and engineering communities. One of the intrinsic features of
tensors is the concept of tensor eigenvalues and eigenvectors
which turns out to be much more complex than that of the square
matrix case (see for example, \cite{CLN, LN, LZ,QL}).

Let us denote
$$
S^{I} = \left\{\textbf{x}=(x_i)_{i \in I}\in \mathbb{R}^{I}\ : \
x_i\geq 0,\quad \sum\limits_{i\in I}x_i = 1\right\}.
$$
By means of $\mathcal P$ one defines an operator $V: S^I\to S^I$
by
\begin{equation}\label{BB}
(V(\xb))_k = \sum\limits_{i_{1},i_2,\dots,i_m\in I}P_{i_1i_2\dots
i_m,k}x_{i_1}x_{i_2}\dots x_{i_m}, \ \ k\in I.
\end{equation}
This operator is called {\it $m$-ordered polynomial stochastic
operator} (in short $m$-ordered PSO). We note that PSO has a
direct connection with non-linear Markov operators which are
intensively studied by many scientists (see \cite{Kol} for recent
review). Therefore, it is crucial to study several properties of
this operator \cite{CLN,LZ,Q}. One of the important ones is its
surjectivity. It turn out that when the set $I$ is finite, in
\cite{ME18} we have established that the sturjectivity of $V$ is
equivalent to its orthogonal preserving property.

We notice that if the order of the operator is 2 (i.e. $m=2$), then the nonlinear Markov
operator $V$ given by \eqref{BB} is called a
\textit{discrete quadratic stochastic operator (DQSO)} which has many
applications in population genetics \cite{J2013}. Note that such kind of
operators is traced back to Bernstein's work \cite{1} where they
appeared from the problems of population genetics. We refer the
reader to \cite{6,11,MG2015} as the exposition of the recent
achievements and open problems in the theory of the QSO can be
further researched. In \cite{FOA,FOA2} the surjectivity of DQSO and its relation with orthogonal preserving property of $V$ have been
investigated.

On the other hand, there has been much interest in recent years in self-organizing search methods
in the q.s.o. field. Recently, in \cite{GSM} it has been considered QSO on the set of $\sigma$-additive measures on $[0, 1]$ (see also \cite{taha}). In particular, points of the unit interval $[0, 1]$ serve to code (continuum valued) traits
attributed to each individual from a considered population. The main aim of the present paper is to investigate projective surjectivity of such kind of
continuum analogous of QSO. First, in Section 3, we revise the results of \cite{FOA}, and in section 4, we apply them
to the projective surjectivity of QSO acting on the set of probability measures. We notice that very particular cases of QSO have been studied (see for example, \cite{BB,BDP,BP,GSM}). In the last section 5, a short application of the main result to the existence of
positive solutions nonlinear Hammerstein integral equations is carried out. Certain Hammerstein integral equations associated with
finite dimensional DQSO has been investigated in \cite{GMS19}. The
obtained results are even new and opens new insight to this topic.

\section{Discrete Quadratic Stochastic Operators}

In this section we give basic notations and some known results from the theory of discrete quadratic stochastic operators.

Let $E$ be a subset of $\mathbb N$ such that $|E|\geq2$. Denote
$$
S^{E} = \left\{\textbf{x}=(x_i)_{i\in E}\in \mathbb{R}^{E}\ : \ x_i\geq 0,\ \forall i\in E\quad\mbox{and}\ \sum\limits_{i\in E}x_i = 1\right\}.
$$
We notice that there is only two case for cardinality of $E$. So, in special cases we write $S$ or $S^{d-1}$ instead of $S^E$ when $E$ is infinite or $|E|=d$ respectively.
In what follows, by $\eb_i$ we denote the standard basis in $S^E$,
i.e. $\eb_i=(\d_{ik})_{k\in E}$ ($i\in E$), where $\d_{ij}$ is the
Kronecker delta.

Let $V$ be a mapping on $S^E$ defined by
\begin{eqnarray}\label{eqn_qso}
(V(\xb))_{k} = \sum\limits_{i,j\in E}P_{ij,k}x_ix_j, \ \ \ \ \forall k\in E,
\end{eqnarray}
here $P_{ij,k}$ are hereditary coefficients which satisfy
\begin{eqnarray}\label{eqn_coef_cond}
P_{ij,k}\geq0, \quad P_{ij,k}=P_{ji,k}, \quad
\sum\limits_{k\in E}P_{ij,k} = 1, \quad \forall i,j,k\in E
\end{eqnarray}
One can see that $V$ maps $S^E$ into itself and $V$ is called
\textit{Discrete Quadratic Stochastic Operator (DQSO)} \cite{M2000}.

By \textit{support} of $ \xb = (x_{i})_{i\in E}\in S^E $ we mean a
set $supp(\xb) = \left\{ i \in E \ : \  x_{i} \neq 0 \right\}$. A
sequence  $\{A_{k}\}_{k\geq1}$ of sets is called \textit{cover} of a set $
B $ if $ \bigcup\limits_{k=1}^{\infty}A_{k} = B$ and $ A_{i}\cap
A_{j} = \emptyset $ for $ i,j \in \bn $ ($i\neq j$).
Recall that two vectors $ \xb,\yb\in S^E $
are called \textit{orthogonal} (denoted by $ \xb \perp \yb $) if $
supp(\xb) \cap supp(\yb) = \emptyset $. If $ \xb,\yb\in S^E$, then
one can see that  $ \xb \perp \yb $ if and only if $ \xb \circ \yb =
0 $. Here, $\xb \circ \yb=\sum_{i\in E}x_iy_i$.

\begin{defn}
    A DQSO $V$ given by \eqref{eqn_qso} is called \textit{orthogonality
        preserving DQSO (OP DQSO)} if for any $ \xb,\yb \in S $ with $ \xb
    \perp \yb $ one has $ V(\xb) \perp V(\yb) $.
\end{defn}

Recall \cite{M2000,Far_Has_Temir} that a DQSO $ V : S^E \rightarrow S^E $
is called \textit{Volterra} if one has
\begin{eqnarray}\label{eqn_cond_Vol}
    P_{ij,k} =0 \textmd{ if } k \notin \{i,j\}, \ \ \forall i,j,k\in E.
\end{eqnarray}

\begin{rem}
In \cite{Far_Has_Temir} it was given an alternative definition
Volterra operator in terms of extremal elements of $S^E$.
\end{rem}

One can check \cite{M2000,Far_Has_Temir} that a DQSO $V$ is Volterra if and only
if one has
$$
(V(\xb))_{k} = x_{k} \left( 1 + \sum\limits_{i\in E}a_{ki}x_{i} \right), \ \ \ \ \ \forall k\in E,
$$
where $ a_{ki} = 2P_{ik,k}-1 $ for all $ i,k \in E $. One can see that $ a_{ki}=-a_{ik} $.
This representation leads us to the following definitions.

\begin{defn}
A DQSO $V : S^{E} \rightarrow S^{E} $ is called $ \pi $-Volterra if there is a permutation $ \pi $ of $E$ such that $ V $ has the following
form
$$
(V(\xb))_{k} = x_{\pi(k)} \left( 1 + \sum\limits_{i\in E}a_{\pi(k)i}x_{i} \right), \ \ \ \ \ \forall k\in E,
$$
where $ a_{\pi(k)i} = 2P_{i\pi(k),k}-1 $, $ a_{\pi(k)i}=-a_{i\pi(k)} $ for any $ i,k \in E$.
\end{defn}

In \cite{Has_Far_OPQSO_e,taha} it has been proved the
following result.

\begin{thm}\label{VV} Let
$|E|=d$ and $ V $ be a DQSO on $ S^{d-1} $. Then the
following statements are equivalent:
\begin{itemize}
 \item[(i)] $ V $ is orthogonality preserving;
 \item[(ii)] $ V $ is surjective;
 \item[(iii)] $ V $ is $ \pi $-Volterra QSO.
\end{itemize}
\end{thm}

We emphasize that there is a big difference between finite and
infinite dimensional settings. It is known \cite{FOA} that in the infinite
dimensional setting, some implication of Theorem \ref{VV} fails.

\begin{thm}\label{VV1}\cite{FOA} Let $ V $ be infinite dimensional DQSO such that $ V(\eb_{i}) = \eb_{\pi(i)} $ for some permutation $\pi$ of $\bn$. Then the
following statements are
    equivalent:
    \begin{itemize}
     \item[(i)] $ V $ is surjective;
        \item[(ii)] $ V $ is orthogonality preserving;
\item[(iii)] $ V $ is  $ \pi- $Volterra QSO.
          \end{itemize}
          \end{thm}

\begin{thm}\cite{FOA}
Let $V$ be a surjective and OP infinite dimensional DQSO. Then $V$ is a $\pi$-Volterra for some permutation
$\pi:\mathbb N\to\mathbb N$.
\end{thm}

\section{Surjectivity of DQSO}

In this section, we are going to provide a sufficient condition for the surjectivity of infinite dimensional DQSOs.

Let $E$ be a subset of $\mathbb N$. We denote
$$
{\bf B}_1^+=\left\{\textbf{x}=(x_i)_{i \in E}\in
\mathbb{R}^{E}\ :\ x_i\geq0,\ \ \ \forall i\in E\ \mbox{and }
\sum_{j\in E}x_j\leq1\right\}.
$$
We can extend each DQSO $V$ from $S^E$ to ${\bf B}_1^+$ by the same formula \eqref{eqn_qso}. Then the following crucial result is true.
\begin{lem}\label{lem1}
Let $V$ be a DQSO on ${\bf B}_1^+$. Then one has
$$
V(S^E)\subset S^E,\ \ \ \ V({\bf B}_1^+\setminus S^E)\subset{\bf B}_1^+\setminus S^E.
$$
\end{lem}
\begin{proof}
For a given $r\geq0$ we denote
$$
S_r^{E} = \left\{\textbf{x}=(x_i)_{i \in E}\in \mathbb{R}^{E}\ :\ x_i\geq0,\ \ \ \forall i\in E\ \mbox{and }
\sum_{j\in E}x_j=r\right\}.
$$
We notice that $S_1^E=S^E$.
Then it is
obvious that
$$
{\bf B}_1^+=\bigcup_{r\in[0,1]}S_r^E.
$$
Take any $r\in[0,1]$ and  an arbitrary $\xb\in S_r^E$. One can check that
$(V(\xb))_k\geq0$ for all $k\geq1$. Furthermore, using \eqref{eqn_coef_cond}
we get
$$
\sum\limits_{k=1}^\infty(V(\xb))_k=\sum\limits_{i,j\in E}x_{i}x_{j}=r^2.
$$
From this we find $V(\xb)\in S_{r^2}^E$. Hence, we infer that
$$
V(S^E)\subset S^E\ \ \mbox{and} \ \ V({\bf B}_1^+\setminus S^E)\subset{\bf B}_1^+\setminus S^E,
$$
which completes the proof.
\end{proof}
\begin{thm}\label{lem2}
Let $V$ be a DQSO on ${\bf B}_1^+$. Then the following statements are equivalent:
\begin{itemize}
\item[(i)] $V$ is surjective on ${\bf B}_1^+$;
\item[(ii)] $V$ is surjective on $S^E$;
\item[(iii)] $V$ is surjective on ${\bf B}_1^+\setminus S^E$.
\end{itemize}
\end{thm}

\begin{proof} Thanks to Lemma \ref{lem1} the implication $(i)\Rightarrow(ii)$ is obvious.

$(ii)\Rightarrow(iii)$. Assume that $V$ is surjective on $S^E$.
Let $r\in(0,1)$, then we define an operator $T_r:S_r^E\to S^E$ as follows $T_r(\xb)=r^{-1}\xb$ for all $\xb\in S_r^E$.
We notice that $T_r$ is a bijection. Then keeping in mind $V(rS^E)=r^2V(S^E)$ and $r^2S^E=S_{r^2}^E$, one gets
$$
V(S_r^E)=V(rT_r(S^E_r))=V(rS^E)=S_{r^2}^E.
$$
From the arbitrariness of $r$ and $V({\bf 0})={\bf 0}$, we find
$$
V\bigg(\bigcup_{r\in[0,1)}S_r^E\bigg)=\bigcup_{r\in[0,1)}S_r^E,
$$
which means $V({\bf B}_1^+\setminus S^E)={\bf B}_1^+\setminus S^E$.

$(iii)\Rightarrow(i)$. One can see that $V(S_r^E)\subset S_{r^2}^E$ for any $r\in[0, 1)$. Due to the surjectivity of $V$ on
$\bigcup_{r\in[0,1)}S_r^E$, we conclude
$$
V(S_r^E)=S_{r^2}^E,\ \ \ \ \forall r\in[0, 1).
$$
Then, for any $r>0$, one has
$$
V(S^E)=V(r^{-1}S_r^E)=r^{-2}V(S_r^E)=r^{-2}S_{r^2}^E=S^E.
$$
The last one together with $V({\bf B}_1^+\setminus S^E)={\bf B}_1^+\setminus S^E$ implies that $V({\bf B}_1^+)={\bf B}_1^+$.
This completes the proof.
\end{proof}

\begin{rem}
Thanks to Theorem \ref{lem2}, to establish the surjectivity of DQSO $V$ on ${\bf B}_1^+$ it is enough to consider it only on $S^E$.
\end{rem}

Let us recall the Cauchy Product which has the following form:
\begin{equation}\label{Cauchy}
\left( \sum\limits_{i=1}^{\ifn} x_{i} \right)^{m} = \sum\limits_{i_1,\dots,i_m\in \bn}x_{i_{1}} \cdots x_{i_{m}},\ \ \ \forall m\in\mathbb N,
\end{equation}
where $\sum_{i=1}^{\ifn} x_{i} < \ifn$.

\begin{thm}\label{prop_2}
    Let $V$ be a surjective DQSO on $S$. Then there exists a
    sequence
    $\{j_k\}_{k\geq1}\subset\mathbb N$
    such that $P_{j_kj_k,k}=1$ for all $k\in\bn$.
\end{thm}

\begin{proof} Let us denote
$$
I_{k} = \left\{ j \in \mathbb N:\ P_{jj, k} = 1 \right\}
$$
    First of all, we  show that surjectivity of $V$ implies $ I_{k}\neq\emptyset$ for any $ k \in\mathbb N$.
    Thanks to the surjectivity of $V$, for every $ k\in\mathbb N$ there is  an $ \xb^{(k)} \in S $ such that
    \begin{eqnarray}\label{eqn_prop_2}
    V(\xb^{(k)}) = \eb_{k}.
    \end{eqnarray}
    Now, we consider two cases.

    \textbf{Case 1:} Let $ \abs{supp(\xb^{(k)})} = 1 $. Then one can find a number $j_k\in\bn$ such that $supp(\xb^{(k)}) = \{j_{k}\} $ and
$$
(V(\xb^{(k)}))_{k} = P_{j_{k}j_{k},k} x_{j_{k}}^2 =1,
$$
which yields that $ P_{j_{k}j_{k},k} =1 $. Hence we get $ j_{k} \in I_{k}$.

    \textbf{Case 2:} Let $|supp(\xb^{(k)})|>1$ and $A:=supp(\xb^{(k)})$. From \eqref{eqn_prop_2} we get
    \begin{eqnarray}
    (V(\xb^{(k)}))_{k} = \sum\limits_{i,j\in A}P_{ij,k}x_{i}x_{j} =1.
    \end{eqnarray}
    Now suppose that there are $\bar{i},\bar{j}\in A$ such that $P_{\bar{i}\bar{j},k}<1$.
    One has
    \begin{eqnarray}
    (V(\xb^{(k)}))_{k} &=& \sum\limits_{i,j\in A}P_{ij,k}x_{i}x_{j} \nonumber \\
    &\leq& \sum\limits_{\{i,j\}\subset A\setminus\{\bar{i},\bar{j}\}}x_{i}x_{j} + P_{\bar{i}\bar{j},k}x_{\bar{i}}x_{\bar{j}} \nonumber \\
    &<& \sum\limits_{i,j\in A}x_{i}x_{j}. \nonumber
    \end{eqnarray}
    Using \eqref{Cauchy} the right side of the last one we get
    $$
    (V(\xb^{(k)}))_{k}<1,
    $$
    which contradicts to \eqref{eqn_prop_2}. So, we conclude that
$$
P_{ij,k} = 1,\ \ \ \ \ \ \forall i,j\in{A}.
$$
    In particular we have $ P_{ii,k} =1 $ for any $ i \in A $. This means that $A \subset I_{k}$. From this we immediately get $I_k\neq\emptyset$
for any $k\geq1$.

Thus, we know that $I_k\neq\emptyset$
for every $k\geq1$. Now, we can define a sequence $\{j_k\}_{k\geq1}$ by $j_k=\inf I_k$.  Due to
construction of $I_k$, we get
$P_{j_kj_k,k}=1$ for all $k\in\mathbb N$. This completes the proof.
\end{proof}

Next result gives a sufficient condition for the surjectivity of DQSO.

\begin{thm}\label{thm_surject}
Let $V$ be a DQSO on $S$. Assume that there exists a
sequence
$\{j_n\}_{n\geq1}\subset\mathbb N$ such that
\begin{equation}\label{nec1}
P_{j_nj_m,k}=0,\ \ \ \ \ \forall k\notin\{n,m\}.
\end{equation}
Then $V$ is surjective.
\end{thm}

\begin{proof}
Let $I:=\{j_n\}_{n\geq1}$ be a subset of $\mathbb N$ for which \eqref{nec1} is satisfied.  Now, we define a new cubic matrix $\tilde{\mathcal P}=(\tilde{P}_{ij,k})_{i,j,k\geq1}$
as follows
$$
\tilde{P}_{ij,k}=\left\{
\begin{array}{ll}
P_{\alpha(i)\alpha(j),k}, & k\in\{i,j\},\\
0, & \mbox{otherwise},
\end{array}
\right.
$$
where $\alpha(k)=j_k$ for all $k\geq1$. We consider a DQSO $\tilde{V}$ is given by
$$
(\tilde{V}(\xb))_k=\sum_{i,j\geq1}\tilde{P}_{ij,k}x_{i}x_{j},\ \ \ \forall\xb\in S.
$$
Due to the construction of the cubic matrix $\tilde{\mathcal P}$ one concludes that $\tilde{V}$ is a Volterra DQSO.
Then, thanks to Theorem \ref{VV1},  the operator $\tilde{V}$ is a surjection on $S$.

Let us denote $S_I=\{\xb\in S: supp(\xb)\subset I\}$. Then it is obvious that
operator $T:S\to S_I$ given by $T(\xb)_k=x_{\alpha(k)}$ is a bijection.
Furthermore, we have
$\tilde{V}=V\circ T$. Keeping in mind that $\tilde{V}$ is surjective and $T$ is bijection,
we infer that $V=\tilde{V}\circ T^{-1}$ is surjective, which completes the proof.
\end{proof}

We stress that unfortunately, we are not able to prove that \eqref{nec1} is a necessary for the surjectivity of $V$. However, all construced examples show it is true. So, we may formulate the following conjecture.

\begin{conj}
Let $V$ be a surjective DQSO on $S$. Then there exists a sequence $\{j_n\}_{n\geq1}\subset\mathbb N$ such that \eqref{nec1} holds.
\end{conj}

\section{Quadratic Stochastic operators on $L^1$ and associated DQSO}

In this section, we consider QSO on $L^1$ and construct associated DQSO.
Let $(X,\cf,\lambda)$ be a measurable space with a $\sigma$-finite measure $\lambda$. By
$L^1(X,\cf,\lambda)$ we define an usual $L^1$ space. We notice that $L^1$ can be identified
with the set of all measures (signed ones) absolutely continuous w.r.t. $\lambda$. Namely,
for every non negative $f\in L^1$ we can define a measure $\mu_f$ by
$$
\mu_f(A)=\int_A fd\lambda,\ \ \ \ \ \forall A\in\Omega.
$$
Therefore, in what follows, we may identify measures with functions and visa versa.

By $S(X)$ we denote the set of all probability measures on $X$ which are absolutely continuous w.r.t. $\lambda$.

Recall that a collection of measurable sets $\mathcal B=\{B_k\}_{k\geq1}$ is called {\it partition of $X$} (w.r.t. $\lambda$) if it satisfies
\begin{enumerate}
  \item[(1)] $X=\bigcup_{k\geq1}B_k$;
  \item[(2)]  $B_i\cap B_j=\emptyset$ for all $i\neq j$;
  \item[(3)] $0<\lambda(B_k)<\infty$ for every $k\geq1$.
\end{enumerate}

We denote by $\mathcal{P}(X)$ the set of all partitions of $X$. Since $\lambda$ is $\sigma$-finite, we infer that $\mathcal{P}(X)\neq\emptyset$.

Let us take $\mathcal B=\{B_k\}_{k\geq1}\in \mathcal{P}(X)$.
For any $\xb\in\ell^1$ we define a measure $\mu_\xb^\mathcal{B}$ on $\cf$ as follows:
\begin{equation}\label{mu_x^B}
\mu_\xb^\mathcal{B}(A)=\sum_{k=1}^\infty\frac{x_k}{\lambda(B_k)}\lambda(A\cap B_k),\ \ \ \ \ \ \forall A\in\cf.
\end{equation}

We notice that $\mu_\xb^{\mathcal B}$ is not probability measure when $\|\xb\|_{\ell^1}\neq1$.
A natural question arises: what kind of $\xb\in\ell^1$ is it true $\mu_\xb^{\mathcal B}\in S(X)$?

Recall that $S=\{\xb\in\ell^1: x_i\geq0,\ \forall i\geq1\ \mbox{and } \parallel\xb\parallel_{\ell^1}=1\}$.
Then the following result holds.

\begin{lem}\label{lemmu_x} Let $\mu_\xb^{\mathcal B}$ be a measure given by \eqref{mu_x^B}. Then $\mu_\xb^{\mathcal B}\in S(X)$
iff $\xb\in S$.
\end{lem}
\begin{proof}
Let us assume that $\mu_\xb^{\mathcal B}\in S(X)$. We have
$\mu_\xb^{\mathcal B}(B_k)=x_k$ for every $k\in\mathbb N$. It yields that $0\leq x_k\leq1$ for every $k\in\mathbb N$.
On the other hand, we obtain
$$
1=\mu_\xb^{\mathcal B}(X)=\mu_\xb^{\mathcal B}(\bigcup_{k\geq1}B_k)=\sum_{k\geq1}x_k,
$$
which implies that $\xb\in S$.

Now we suppose that $\xb\in S$. Then we get $\mu_\xb^{\mathcal B}(X)=\sum_{k\geq1}x_k=1$. This means that $\mu_\xb^{\mathcal B}$ is a probability
measure on $X$. Moreover, it is obvious that measure given by \eqref{mu_x^B} is absolutely continuous w.r.t. $\lambda$. Hence, we infer that
$\mu_\xb\in S(X)$.
\end{proof}

\begin{rem}\label{rem_ss}
For a given partition $\mathcal B$ of $X$, thanks to Lemma \ref{lemmu_x} there exists a one-to-one correspondence between $S$ and $M(X,\mathcal B):=\{\mu_\xb^{\mathcal B}\in S(X):\ \xb\in\ell^1\}$.
In other words, every $\mu\in M(X,\mathcal B)$ is uniquely defined by the values $\mu(B_k)$, $k\geq1$.
\end{rem}

\begin{prop} Let $\mathcal B\in \mathcal{P}(X)$.
Then $M(X,\mathcal B)$ is a convex and closed set w.r.t. strong convergence.
Moreover, $M(X,\mathcal B)$ is not compact w.r.t. weak convergence.
\end{prop}

\begin{proof}
One can see that $T: S\to M(X,\mathcal B)$ given by $T\xb=\mu_{\xb}^{\mathcal B}$ is a bijection.
Then any sequence on $M(X,\mathcal B)$ is defined by a sequence $\{\xb^{(n)}\}_{n\geq1}\subset S$.
It is obvious that if $\xb^{(n)}\stackrel{\norm{\cdot}_{\ell_1}}{\longrightarrow}\xb$ then
$\lim\limits_{n\to\infty}\mu_{\xb^{(n)}}^{\mathcal B}(A)=\mu_{\xb}^{\mathcal B}(A)$
for all $A\in\cf$.

Let us pick
a sequence $\{\mu_{\xb^{(n)}}^{\mathcal B}\}_{n\geq1}\subset M(X,\mathcal B)$. Assume that $\mu(A)=\lim\limits_{n\to\infty}\mu_{\xb^{(n)}}^{\mathcal B}(A)$
for every $A\in\cf$. Then we have $\mu\in S(X)$.
On the other hand, we obtain
$$
\mu(B_k)=\lim_{n\to\infty}x_k^{(n)},\ \ \ \ \ \forall k\geq1.
$$
The last one together with $\mu(X)=1$ implies that $\xb^{(n)}$ converges on $S$ w.r.t. $\ell^1$-norm. Hence,
we conclude that $T$ is a homeomorphism. Then due to closedness and convexity of $S$ we infer that $M(X,\mathcal B)$ has the same topological
properties. We notice that $S$ is not compact w.r.t. $\ell^1$-norm. Consequently, $M(X,\mathcal B)$ is not a compact w.r.t. weak convergence.
\end{proof}

\begin{lem}\label{chrylem}
Let $\tilde S(X)=\left\{\mu_f\in S(X): f\ \mbox{is a simple function on}\ L^1\right\}$. Then
\begin{equation}\label{sample=mu}
\tilde{S}(X)=\bigcup_{\mathcal B\in \mathcal{P}(X)}M(X,\mathcal B).
\end{equation}
\end{lem}

\begin{proof}
It is clear that
$\bigcup\limits_{\mathcal{B}\in \mathcal{P}(X)}M(X,\mathcal{B})\subset\tilde S(X)$. Indeed, for
any $\mu_\xb^{\mathcal B}$ we define a simple function
$$
f_{\mu_\xb^{\mathcal B}}(u)=\frac{x_k}{\lambda(B_k)},\ \ \ \ \ \forall u\in B_k,\ \forall k\geq1,
$$
which satisfies
$$
\mu_\xb^{\mathcal B}(A)=\int_A f_{\mu_\xb^{\mathcal B}}d\lambda,\ \ \ \ \ \forall A\in\cf.
$$

Now, we take an arbitrary $\mu_f\in\tilde S(X)$. Then for any $i\geq1$ we have a measurable set $A_i=\left\{u\in X: f(u)=y_i\right\}$.
One may assume that $\lambda(A_i)>0$ for every $i\geq1$. We notice that if $\lambda(A_i)<\infty$ for each $i\in\mathbb N$ then
$\mathcal A=\{A_i\}_{i\geq1}$ is a partition of $X$ and $\mu_f=\mu_\xb^{\mathcal A}$, where $\xb=(\mu_f(A_1),\mu_f(A_2),\dots)\in S$.

If $\lambda(A_j)=\infty$ for some $j\geq1$ then one has $y_j=0$ (otherwise $f$ is not integrable). Hence, $\mu_f(A_j)=0$. So, without loss of generality we may assume that
$y_1=0$, $\lambda(A_1)=\infty$ and $y_i>0$, $\lambda(A_i)<\infty$ for any $i>1$. Pick any partition $\{B_k\}_{k\geq1}$ of $X$ and define
a new partition $\tilde{\mathcal B}$ of $X$ as follows:
$$
\tilde{B}_k=\left\{
\begin{array}{ll}
A_1\cap B_{\frac{k+1}{2}}, & \mbox{if } k\ \mbox{is odd},\\
A_{\frac{k+2}{2}}, & \mbox{if } k\ \mbox{is even}.
\end{array}
\right.
$$
Then $\mu_f=\mu_{\tilde{\yb}}^{\tilde{\mathcal B}}$, where coordinates of $\tilde{\yb}\in S$ are given by
$$
\tilde{y}_k=\left\{
\begin{array}{ll}
0, & \mbox{if } k\ \mbox{is odd},\\
\mu_f(A_{\frac{k+2}{2}}), & \mbox{if } k\ \mbox{is even}.
\end{array}
\right.
$$
The arbitrariness of $\mu_f$ yields that $\tilde{S}(X)\subset\bigcup_{\mathcal{B}\in\mathcal{P}(X)}M(X,\mathcal B)$.
The last one together with $\bigcup_{\mathcal{B}\in\mathcal{P}(X)}M(X,\mathcal B)\subset\tilde S(X)$
implies \eqref{sample=mu}.
\end{proof}

Due to the density argument, from Lemma
\ref{chrylem} we immediately infer the following result.

\begin{prop}\label{prop_isbot}
One has $S(X)=\overline{\bigcup_{\mathcal{B}\in\mathcal P(X)}M(X,\mathcal B)}$, here
the closure in sense of weak convergence.
\end{prop}

\subsection{Projective surjectivity of QSO}

Now, we consider a measurable function $P: X\times X\times\cf\to[0; 1]$ which satisfies the following conditions:
\begin{equation}\label{P(x,y,A)}
P(u,v,A)=P(v,u,A),\ \ \ \ \ \forall u,v\in X,\ \ \forall A\in\cf,
\end{equation}
\begin{equation}\label{P(x,y)}
P(u,v,\cdot)\in S(X),\ \ \ \ \ \forall u,v\in X.
\end{equation}
This function is called {\it transition kernel}, and defines a \textit{Quadratic Stochastic Operator} (in short QSO) by
\begin{equation}\label{Vasos}
(\mathcal V\mu)(A)=\int_X\int_X P(u,v,A)d\mu(u)d\mu(v),\ \ \ \ \forall \mu\in S(X),\ \forall A\in\cf.
\end{equation}
Clearly, $\mathcal V$ maps $S(X)$ into itself.

\begin{defn}
A QSO $\mathcal V$ given by \eqref{Vasos} is called {\it projective surjection} if it is surjective on $M(X, \mathcal B)$
for some $\mathcal B\in\mathcal P(X)$.
\end{defn}

Now we are going to find projective surjection QSOs.
For a given QSO $\mathcal V$ we associate DQSO (this DQSO depends on partition $\{B_k\}_{k\geq1}$) $V_\mathcal{B}: S\to S$
by
\begin{equation}\label{V_mathcalb}
\left(V_\mathcal{B}(\xb)\right)_k=\sum_{i,j\geq1}P_{ij,k}^\mathcal{B}x_ix_j,\ \ \ \ \ \ \ \forall k\geq1,
\end{equation}
where
\begin{equation}\label{P_ijkasos}
P_{ij,k}^\mathcal{B}=\frac{1}{\lambda(B_i)\lambda(B_j)}\int_{B_i}\int_{B_j}P(u,v,B_k)d\lambda(u)d\lambda(v),\ \ \ \ \ \ \ \forall k\geq1.
\end{equation}

\begin{lem}\label{lemV=V_B}
Let $\mathcal B=\{B_k\}_{k\geq1}\in\mathcal P(X)$. Then for every $\xb\in S$ it holds
$$
(\mathcal V\mu_{\xb}^{\mathcal B})(B_k)=(V_{\mathcal B}(\xb))_k,\ \ \ \ \ \forall k\geq1.
$$
\end{lem}

\begin{proof}
Let $\mathcal B=\{B_k\}_{k\geq1}$ be a partition of $X$ and $\xb\in S$. Then for any $k\geq1$ we have
\begin{eqnarray*}
(\mathcal V\mu_\xb^{\mathcal B})(B_k)&=&\int_{X}\int_{X}P(u,v,B_k)d\mu_\xb^{\mathcal B}(u)d\mu_\xb^{\mathcal B}(v)\\
&=&\sum_{i,j\geq1}\frac{x_ix_j}{\lambda(B_i)\lambda(B_j)}\int_{B_i}\int_{B_j}P(u,v,B_k)d\lambda(u)d\lambda(v)\\
&=&\sum_{i,j\geq1}P_{ij,k}^\mathcal{B}x_ix_j\\
&=&(V_{\mathcal B}(\xb))_k.
\end{eqnarray*}
\end{proof}

\begin{prop}\label{prop_ss}
Let $\mathcal V$ be a QSO given by \eqref{Vasos} and $\mathcal B\in\mathcal P(X)$. If $P(u,v,\cdot)\in M(X,\mathcal B)$ for every $(u,v)\in X^2$
then $\mathcal V(M(X,\mathcal B))\subset M(X,\mathcal B)$.
\end{prop}

\begin{proof} Let $\mathcal B=\{B_k\}_{k\geq1}$ be partition of $X$. Assume that
$P(u,v,\cdot)\in M(X,\mathcal B)$ for every $(u,v)\in X^2$.
Then for arbitrary $(u,v)\in X^2$ we obtain
\begin{equation}\label{kkeq}
P(u,v,A_k)=\frac{\lambda(A_k)}{\lambda(B_k)}P(u,v,B_k),\ \ \ \ \forall A_k\subset B_k,\ \forall k\geq1.
\end{equation}
For any $\xb\in S$ we define $\yb\in S$ as follows $y_k=(V_\mathcal{B}(\xb))_k$, $k\geq1$. The due to Lemma \ref{lemV=V_B}
we get
$$
(\mathcal V\mu_{\xb}^{\mathcal B})(B_k)=y_k,\ \ \ \ \ \forall k\geq1.
$$
Let us establish that $\mathcal V\mu_\xb^\mathcal{B}=\mu_\yb^\mathcal{B}$. Take
an arbitrary measurable $A\in\cf$ and denote $A_k=A\cap B_k$ for every $k\geq1$.
Keeping in mind \eqref{kkeq} one gets
\begin{eqnarray*}
(\mathcal V\mu_\xb^{\mathcal B})(A)&=&\int_{X}\int_{X}P(u,v,A)
d\mu_\xb^{\mathcal B}(u)d\mu_\xb^{\mathcal B}(v)\\
&=&\sum_{i,j\geq1}\frac{x_ix_j}{\lambda(B_i)\lambda(B_j)}\sum_{k\geq1}
\int_{B_i}\int_{B_j}P(u,v,A_k)d\lambda(u)d\lambda(v)\\
&=&\sum_{i,j\geq1}\frac{x_ix_j}{\lambda(B_i)\lambda(B_j)}\sum_{k\geq1}
\frac{\lambda(A_k)}{\lambda(B_k)}\int_{B_i}\int_{B_j}P(u,v,B_k)d\lambda(u)d\lambda(v)\\
&=&\sum_{k\geq1}\frac{\lambda(A_k)}{\lambda(B_k)}y_k\\
&=&\mu_\yb^{\mathcal B}(A),
\end{eqnarray*}
which yields $\mathcal V\mu_\xb^\mathcal{B}=\mu_\yb^\mathcal{B}$. The arbitrariness of $\xb\in S$ implies $\mathcal V(M(X,\mathcal B))\subset M(X,\mathcal B)$.
The proof is complete.
\end{proof}

Now we are going to find sufficiently conditions of projective surjectivity of QSO given by \eqref{Vasos}.

\begin{thm}\label{thm8}
Let $(X, \cf, \lambda)$ be a measurable space with $\sigma$-finite measure $\lambda$
and $\mathcal B=\{B_k\}_{k\geq1}\in\mathcal P(X)$. Assume that the transition kernel $P$ of QSO $\mathcal V$ satisfies the followings conditions:
\begin{itemize}
\item[(i)] $P(u,v,\cdot)\in M(X,\mathcal B),\ \ \ \forall (u,v)\in X^2$;
\item[(ii)] there exists a sequence $\{j_n\}_{n\geq1}\subset\mathbb N$ such that
\begin{equation}\label{eqn_P(u,v,B_k)}
P(u,v,B_k)=\frac{\lambda(B_k\cap B_{n})}{2\lambda(B_{n})}+\frac{\lambda(B_k\cap B_{m})}{2\lambda(B_{m})},
\ \ \ \ \forall (u,v)\in B_{j_n}\times B_{j_m},\ \forall k\in\mathbb N.
\end{equation}
\end{itemize}
Then $\mathcal V$ is projective surjection.
\end{thm}

\begin{proof}
Assume that all conditions of the theorem hold. From the condition (i), according to Proposition \ref{prop_ss}
we have $\mathcal V: M(X,\mathcal B)\to M(X,\mathcal B)$.

Now, let us show that $\mathcal V(M(X,\mathcal B))=M(X,\mathcal B)$. For any triple $(n,m,k)\in\mathbb N^3$
from \eqref{eqn_P(u,v,B_k)} after simple calculations, we get
$$
P_{j_nj_m,k}^{\mathcal B}=P_{j_mj_n,k}^{\mathcal B}=\left\{
\begin{array}{ll}
1, & \mbox{if }\ n=m=k;\\[1mm]
\frac{1}{2}, & \mbox{if }\ n=k\neq m;\\[1mm]
0, & \mbox{if }\ k\notin\{n,m\}.
\end{array}
\right.
$$

Then, due to Theorem \ref{thm_surject}, the corresponding DQSO $V_\mathcal{B}$ is a surjection. Hence, for any $\yb\in S$
we can find $\xb\in S$ such that $V_\mathcal{B}(\xb)=\yb$. Consequently,
by Lemma \ref{lemV=V_B} one has
$$
(\mathcal V\mu_\xb^{\mathcal B})(B_k)=\mu_\yb^{\mathcal B}(B_k),\ \ \ \ \ \forall k\in\mathbb N.
$$
From the last one, keeping in mind $\mathcal V\mu_\xb^{\mathcal B}\in M(X,\mathcal B)$ thanks to Remark \ref{rem_ss} we infer
$$
\mathcal V\mu_\xb^{\mathcal B}=\mu_\yb^{\mathcal B}.
$$
Finally, the arbitrariness of $\yb\in S$ yields that $\mathcal V(M(X,\mathcal B))=M(X,\mathcal B)$. This completes the proof.
\end{proof}

\begin{rem}
We notice that the conclusion of the last theorem will be true if \eqref{eqn_P(u,v,B_k)} holds almost everywhere in $B_{j_n}\times B_{j_m}$.
\end{rem}

For any measurable set $A\subset X$ we define
$$
\mathcal E_A=\left\{(x,y)\in (X\setminus A)^2:\ P(x,y,A)\neq0\right\}.
$$

\begin{thm}\label{Thm9}
Let $(X, \cf, \lambda)$ be a measurable space with $\sigma$-finite measure $\lambda$
and $\mathcal B=\{B_k\}_{k\geq1}\in\mathcal P(X)$. Then there is only one QSO $\mathcal V$ whose transition kernel satisfies the followings:
\begin{itemize}
\item[(i)] $P(u,v,\cdot)\in M(X,\mathcal B),\ \ \ \forall (u,v)\in X^2$;
\item[(ii)] $\lambda(\mathcal E_{B_k})=0$ for every $k\in\mathbb N$.
\end{itemize}
Moreover, $\mathcal V$ is projective surjection.
\end{thm}

\begin{proof}
From (i) we have $P(u,v,\cdot)=\mu_{\xb(u,v)}^\mathcal B$ for any $u,v\in X$. Without loss of generality we may replace the condition (ii)
to the following one
$$
\mu_{\xb(u,v)}^\mathcal{B}(B_k)=0,\ \ \ \ \ \forall (u,v)\in (X\setminus B_k)^2,\ \forall k\in\mathbb N.
$$
Then for any $(u,v)\in B_n\times B_m$ we obtain
$$
\mu_{\xb(u,v)}^\mathcal{B}(B_n)+\mu_{\xb(u,v)}^\mathcal{B}(B_m)=1.
$$
Keeping in mind $\xb(u,v)=\xb(v,u)$, from the last one we have $\xb(u,v)=\frac{1}{2}{\bf e}_n+\frac{1}{2}{\bf e}_m$ for every
$(u,v)\in B_n\times B_m$. Hence, the transitional kernel has the following form
\begin{equation}\label{trkerP}
P(u,v,A)=\frac{\lambda(A\cap B_n)}{2\lambda(B_n)}+\frac{\lambda(A\cap B_m)}{2\lambda(B_m)},\ \ \ \ \forall (u,v)\in B_n\times B_m,\ \forall A\in\cf.
\end{equation}

We notice that \eqref{trkerP} implies \eqref{eqn_P(u,v,B_k)} for the sequence $\{n\}_{n\geq1}$. Then, due to Theorem \ref{thm8} we conclude that
$\mathcal V$ is projective surjection.
The proof is complete.
\end{proof}

\begin{cor}
Let $(X, \cf, \lambda)$ be a measurable space with $\sigma$-finite measure $\lambda$
and $\mathcal B=\{B_k\}_{k\geq1}\in\mathcal P(X)$. There is only one CQSO $\mathcal V$ whose transition kernel satisfies the followings:
\begin{itemize}
\item[(i)] $P(u,v,\cdot)\in M(X,\mathcal B),\ \ \ \forall u,v\in X$;
\item[(ii)] $\lambda(\mathcal E_{A})=0$ for every $A\in\cf$.
\end{itemize}
Moreover, $\mathcal V$ is projective surjection.
\end{cor}

\section{Application}

In this section we are give a direct application of the projective surjectivity of QSO to the existence
of positive solutions of certain nonlinear integral equations.

Let $(X,\cf,\lambda)$ be a measurable space with a $\sigma$-finite measure $\lambda$.
Let us consider the following nonlinear
Hammerstein integral equation:
\begin{equation}\label{inteq}
\int_X\int_X K(u,v,t)x(u)x(v)d\lambda(u)d\lambda(v)=\varphi(t),
\end{equation}
where $K$ is some positive kernel and $\varphi\in L^1_+$ is a given function.

We note that this type of equation appeared in several problems of
astrophysics, mechanics, and biology. Here in the equation,
$K:X^3\to\br$ and $\varphi: X\to\br$ are given functions, and
$x:\O\to\br$ is an unknown one. Generally speaking, in order to
solve the nonlinear Hammerstein integral equation
\eqref{inteq} over some functions space,  one should impose
on some constraints on $K(\cdot,\cdot,\cdot)$. There are several
works where the existence of solutions the above given equation
have been carried out by means of contraction methods (see
\cite{Atkinson,Ban,Krasnoselskii,Some}).
In this section, we are going to another approach for the
existence of positive solutions of \eqref{inteq}. In what follows, we
consider the equation \eqref{inteq} over $L^1$-spaces.

Multiplying \eqref{inteq} by a function $g$ from
$L^\infty$ and integrating it, we obtain
\begin{equation}\label{eqint}
\int_X\int_X\int_X g(t)K(u,v,t)x(u)x(v)d\lambda(u)d\lambda(v)d\lambda(t)=\int_X g(t)\varphi(t)d\lambda(t).
\end{equation}

We stress that the arbitrariness of $g$ implies that \eqref{eqint} and \eqref{inteq} are equivalent.

Now, assume that there is a transition kernel $P$ such that
$$
\int_X\int_X\int_X g(t)K(u,v,t)x(u)x(v)d\lambda(u)d\lambda(v)d\lambda(t)=
\int_X\int_X\int_X g(t)P(u,v,dt)x(u)x(v)d\lambda(u)d\lambda(v)
$$
for all $x\in L^1$ and $g\in L^\infty$.

Then \eqref{eqint} is reduced to
$$
\int_X\int_X\int_X g(t)P(u,v,dt)x(u)x(v)d\lambda(u)d\lambda(v)=\int_X g(t)\varphi(t)d\lambda(t).
$$
Now, taking $g=\chi_A$, $A\in\cf$, we arrive at
\begin{equation*}
(\mathcal V\mu_x)(A)=\mu_\varphi(A),
\end{equation*}
where, as before, $\mu_x(A)=\int_A x(u)d\lambda(u)$. Assume that
$$
\int_X \varphi d\lambda=1.
$$
Hence, the integral equation \eqref{inteq} is reduced to the equation
\begin{equation}\label{eqint3}
\mathcal V\mu=\mu_\varphi,
\end{equation}
where $\mu\in S(X)$.

Hence, the following result is true.
\begin{thm}
Let $(X, \cf, \lambda)$ be a measurable space with $\sigma$-finite measure $\lambda$
and $\mathcal B=\{B_k\}_{k\geq1}\in\mathcal P(X)$. Assume that QSO $\mathcal V$ is a projective surjection on $M(X,\mathcal B)$. Then for any $\mu_\varphi\in M(X,\mathcal B)$
the equation \eqref{eqint3} has a solution in $M(X,\mathcal B)$.
\end{thm}

\section*{Acknowledgments}
The present work is supported by the UAEU UPAR Grant No. G00003447. The first named author (A.F.E.) acknowledges the Ministry of Higher Education (MOHE) and Research Management Centre-UTM, Universiti Teknologi Malaysia (UTM) for the financial support through the research grant (vote number 17J93).

\end{document}